\documentclass[]{interact}

\usepackage{epstopdf}% To incorporate .eps illustrations using PDFLaTeX, etc.
\usepackage[caption=false]{subfig}% Support for small, `sub' figures and tables

\usepackage[numbers,sort&compress]{natbib}% Citation support using natbib.sty
\bibpunct[, ]{[}{]}{,}{n}{,}{,}% Citation support using natbib.sty
\makeatletter% @ becomes a letter
\def\NAT@def@citea{\def\@citea{\NAT@separator}}% Suppress spaces between citations using natbib.sty
\makeatother% @ becomes a symbol again

\theoremstyle{plain}% Theorem-like structures provided by amsthm.sty
\newtheorem{theorem}{Theorem}[section]

\newtheorem{proposition}[theorem]{Proposition}

\theoremstyle{definition}
\newtheorem{definition}[theorem]{Definition}
\newtheorem{example}[theorem]{Example}

\theoremstyle{remark}
\newtheorem{remark}{Remark}

\begin{document}

\articletype{ARTICLE TEMPLATE}% Specify the article type or omit as appropriate

\title{Approximation of an algebra by evolution algebras}

\author{
\name{A.~N. Imomkulov\textsuperscript{a}\thanks{CONTACT A.~N. Imomkulov. Email: aimomkulov@gmail.com} and U.~A. Rozikov\textsuperscript{a}}
\affil{\textsuperscript{a}Institute of Mathematics, 81 Mirzo Ulug'bek str., Tashkent, Uzbekistan.}
}

\maketitle

\begin{abstract}
It is known that any multiplication of a finite dimensional algebra is
determined by a matrix of structural constants.
In general, this is a cubic matrix. Difficulty of investigation of an algebra depends on the cubic matrix.
Such a cubic matrix defines a quadratic mapping called an evolution operator.
In the case of evolution algebras, the cubic matrix consists many zeros allowing to reduce it
to a square matrix.
 In this paper for any finite dimensional algebra we construct a family of evolution algebras
corresponding to Jacobian of the evolution operator at a point of the algebra. We obtain some results answering
 how properties of an algebra depends on the properties of the corresponding family of evolution algebras.
Moreover, we consider evolution algebras corresponding
to 2 and 3-dimensional nilpotent Leibniz algebras. We prove that such evolution algebras are
nilpotent too. Also we classify such evolution algebras.
\end{abstract}


\section{Introduction}\label{sec:intro}

Let $(A,\cdot)$ be an algebra over a field $K$. If it admits a basis
 $e_{1},e_{2},\dots,e_{n},\dots$, such that
 $$\begin{array}{cc}
 e_{i}\cdot e_{j}=0, \,\, \mbox{if} \,\,  i\neq j\\[2mm]
 e_{i}\cdot e_{i}=\sum_{k} a_{ik}e_{k}, \,\, \mbox{for any} \,\, i
\end{array}$$
then it is called an evolution algebra. This basis is called a natural basis.

We note that to every evolution algebra corresponds a \emph{square} matrix
$(a_{ik})$ of structural constants. This makes simpler the investigation
of an evolution algebra compared to a general algebra. Because, a multiplication of a general algebra can be given
by a \emph{cubic} matrix.

 In \cite{11} the following basic properties of evolution algebras are proved:
 they are  commutative, flexible; are not power-associative and not associative, in general.
The direct sum of evolution algebras is also an evolution algebra.
For recently obtained results related to evolution algebras  see \cite{1},
\cite{12}, \cite{13}, \cite{CLOR1}, \cite{CLR}, \cite{3}-\cite{5},  \cite{8}, \cite{ORV}, \cite{9}  and the references therein.

In \cite{8} the dynamics of absolute
nilpotent and idempotent elements in chains generated by two-dimensional
evolution algebras are studied. In \cite{2} the authors consider
an evolution algebra which has a rectangular matrix of structural constants.
This algebra is called evolution algebras of  \,
\textquotedblleft chicken\textquotedblright\,  population (EACP).
The classification of three-dimensional complex EACPs is obtained.
Moreover, some $(n + 1)$-dimensional EACPs are described.

Thus the theory of evolution algebras is well developed.
In the study of nonlinear dynamical systems, a linearization is approximation of the nonlinear mapping by the linear one, that is defined by the Jacobian
 of the nonlinear mapping at a fixed point. In this paper we use this linearization for
 a quadratic operator corresponding to cubic matrix of structural constants of a finite
 dimensional algebra. By the linear operators (corresponding to Jacobian of the quadratic operator at a point of the algebra)
 we construct a family of evolution algebras.
We give relations between given general algebra and corresponding to it evolution algebras.

\section{Approximation of finite dimensional algebras with evolution algebras}\label{sec:two}

Given a field $K$, any finite-dimensional algebra can be specified up
to isomorphism by giving its dimension (say $m$), and specifying $m^3$ structure constants $\gamma_{ij,k}$, which are scalars.
These structure constants determine the multiplication in $\mathcal A$ via the following rule:
\[{e}_{i} {e}_{j} = \sum_{k=1}^m \gamma_{ij,k} {e}_{k},\]
where $e_1,\dots,e_m$ form a basis of $\mathcal A$.

Thus the multiplication of a finite-dimensional algebra is given by a cubic matrix $(\gamma_{ij,k})$.
This matrix also defines an evolution operator $F: K^m\to K^m$ as
$$F:x_{k}'=\sum_{i,j}\gamma_{ij,k}x_{i}x_{j}, \ \ k=1,\dots, m.$$
The investigation of the algebra $\mathcal A$ and the quadratic operator $F$ is difficult,
since they are determined by cubic matrices.
 While this investigation can be simplified by taking some values of $\gamma_{ij,k}$ equal to zero.
Non-linear functions, like $F$,  mainly reduced to linear functions (linearization) by Jacobian
to prove the local and some global results (see for example \cite{Ga}, \cite{Rb}).

Jacobian of this operator is the following:
$$J_{F}(x)=\left(\sum_{j}\gamma_{pj,k}x_{j}+\sum_{i}\gamma_{ip,k}x_{i}\right)_{p,k=1}^n=
\left(\sum_{i=1}^n(\gamma_{pi,k}+\gamma_{ip,k})x_{i}\right)_{p,k=1}^n.$$
Since this is a square matrix we can define an evolution algebra with structural
constants matrix $J_{F}(x)$. To do this, denote
\begin{equation}\label{svyaz}
\beta_{pk}(x)=\sum_{i}(\gamma_{pi,k}+\gamma_{ip,k})x_{i}
\end{equation}
and define an evolution algebra $E_{x}=<\tilde{e_{1}},\tilde{e_{2}},...,\tilde{e_{n}}>$ by multiplication
$$\tilde{e_{i}}\tilde{e_{j}}=0, \ \ \tilde{e_{i}}^2=\sum_{k}\beta_{ik}(x)\tilde{e_{k}}.$$

Note that this kind of evolution algebra first was defined in \cite{10} for an algebra
of mosquito population.

Thus (\ref{svyaz}) is a relation between algebras $\mathcal A$ and $E_{x}$.
Similarly to the theory of dynamical systems we call $E_x$ the {\it approximation} of
$\mathcal A$ at point $x$.

%\section{For the given evolution algebra...}

{\bf The main question } of this section is  ``What can we say about algebra $\mathcal A$
with matrix of structural constants $(\gamma_{pi,k})_{p,i,k=1}^{n}$
knowing properties of the evolution algebra $E_x$ for some $x\in \mathcal A$ with the relation (\ref{svyaz})?"

%Let $E$ be evolution algebra with matrix of structural
%constants $(a_{pk})$. We write the relation by formula (\ref{svyaz}):
%\begin{equation}\label{svyaz1}
%a_{pk}=\sum_{i}(\gamma_{pi,k}+\gamma_{ip,k})x_{i}.
%\end{equation}
%
%\vskip10pt

For an algebra $\mathcal A$ with matrix of structural constants
$M_{\mathcal A}=(\gamma_{pi,k})_{p,i,k=1}^{n}$, denote
$$I_{\mathcal A}=\{p\in\{1,2,...,n\}: \det\Gamma_{p}\neq0\},$$
where $\Gamma_{p}=\left(\gamma_{pi,k}+\gamma_{ip,k}\right)_{i,k=1}^n$.

 Fix $x\in \mathcal A$ and denote by $a_{p}\equiv a_p(x)$ the vector
considered as $p$-th column of matrix $(\beta_{pk}(x))$ defined in (\ref{svyaz}).

\begin{theorem}\label{mavjudlik} For a given non-trivial evolution algebra $E$ with
$\mathcal M_{E}=\left(a_{pk}\right)$ and an algebra $\mathcal A$ with
$\mathcal M_{\mathcal A}=\left(\gamma_{pi,k}\right)$ there is an element $x\in \mathcal A\backslash \{0\}$ satisfying (\ref{svyaz})
if and only if
\begin{equation}\label{umumiyshart1}
\Gamma_{p}^{-1}\cdot a_{p}=\Gamma_{q}^{-1}\cdot a_{q}, \,\,\, for\,\,\, any\,\,\, p,q\in I_{\mathcal A}
\end{equation}
and
\begin{equation}\label{umumiyshart2}
{\rm rank}\, \Gamma_{p}={\rm rank}\left(\Gamma_{p},a_{p}\right),\,\,\, for\,\,\, any\,\,\, p=\{1,2,...,n\}\backslash I_{\mathcal A}.
\end{equation}
\end{theorem}

\begin{proof}
We must answer for the question: When does the system of equations (\ref{svyaz}) have a solution
relatively unknown $x=(x_{1},...,x_{n})$? In this case we have $n^2$ equations and $n$ unknowns.
By theory of system of linear equations it is easy to see that $x=\Gamma_{p}^{-1}\cdot a_{p},\,\, p\in I_{\mathcal A}$
and so  $\Gamma_{p}^{-1}\cdot a_{p}=\Gamma_{q}^{-1}\cdot a_{q}$ for any $p,q\in I_{\mathcal A}$.
If $p\in \{1,2,...,n\}\backslash I_{\mathcal A}$ then we know that
${\rm rank}\,\Gamma_{p}={\rm rank}\, \left(\Gamma_{p},a_{p}\right)$ holds.
\end{proof}

Let us give some examples checking conditions of Theorem \ref{mavjudlik}.
\begin{example} Let $A$ be the algebra with the (cubic) matrix of structural constants
$$\left( \begin{array}{cccc}
1 & 2\\
0 & 1 \end{array}\Bigg|
\begin{array}{cccc}
0 & 2\\
1 & 1 \end{array}\right)$$
and the evolution algebra $E$ with matrix of structural constants
$$\left( \begin{array}{cccc}
1 & 1\\
-1 & -1 \end{array}\right).$$ For these algebras the system of equations (\ref{svyaz}) has the following form
\begin{equation}\label{misol1}\begin{cases}
2x_{1}=1\\
4x_{1}+3x_{2}=1\\
2x_{2}=-1\\
3x_{1}+2x_{2}=-1 \end{cases}.\end{equation}
In this case $\Gamma_{1}=\left(\begin{array}{cccc}
2 & 0\\
4 & 3 \end{array}\right)$ and
$\Gamma_{2}=\left(\begin{array}{cccc}
0 & 2\\
3 & 2 \end{array}\right)$ and $\det\Gamma_{1}\cdot\det\Gamma_{2}\neq 0$. Therefore
$\Gamma_{1}\cdot a_{1}\neq\Gamma_{2}\cdot a_{2}$, where $a_{1}=\left(\begin{array}{cc}1\\1\end{array}\right)$,
 $a_{2}=\left(\begin{array}{cc}-1\\-1\end{array}\right)$. Namely,  the conditon (\ref{umumiyshart1}) doesn't hold.
 Consequently, system of equations (\ref{misol1}) has no solution.
\end{example}

\begin{example} Let $A$ be the algebra with matrix of structural constants
$$\left( \begin{array}{cccc}
1 & 2\\
0 & 1 \end{array}\Bigg|
\begin{array}{cccc}
0 & -1\\
1 & 1 \end{array}\right)$$
and the evolution algebra $E$ with matrix of structural constants
$\left( \begin{array}{cccc}
1 & 1\\
-1 & -1 \end{array}\right)$. For these algebras system of equations (\ref{svyaz}) has the following form
\begin{equation}\label{misol2}\begin{cases}
2x_{1}=1\\
4x_{1}=1\\
2x_{2}=-1\\
2x_{2}=-1 \end{cases}.\end{equation}
In this case $\Gamma_{1}=\left(\begin{array}{cccc}
2 & 0\\
4 & 0 \end{array}\right)$,
$\Gamma_{2}=\left(\begin{array}{cccc}
0 & 2\\
0 & 2 \end{array}\right)$ and $\det\Gamma_{1}=\det\Gamma_{2}=0$. Then
${\rm rank}\,\Gamma_{1}\neq {\rm rank}\,(\Gamma_{1},a_{1})$ although
 ${\rm rank}\, \Gamma_{2}\neq {\rm rank}\, (\Gamma_{2},a_{2})$, where
 $a_{1}=\left(\begin{array}{cc}1\\1\end{array}\right)$,
 $a_{2}=\left(\begin{array}{cc}-1\\-1\end{array}\right)$. Namely  the conditon (\ref{umumiyshart2}) doesn't hold.
 Consequently, system of equations (\ref{misol2}) has no solution.
\end{example}
\begin{example}
Let $A$ be the algebra with matrix of structural constants
$\left( \begin{array}{cccc}
1 & 2\\
0 & 1 \end{array}\Bigg|
\begin{array}{cccc}
-1 & 0\\
1 & 2 \end{array}\right)$
and the evolution algebra $E$ with matrix of structural constants
$\left( \begin{array}{cccc}
1 & 5\\
1 & 5 \end{array}\right)$. For these algebras the system of equations (\ref{svyaz}) has the following form
\begin{equation}\label{misol3}\begin{cases}
2x_{1}-x_{2}=1\\
4x_{1}+x_{2}=5\\
-x_{1}+2x_{2}=1\\
x_{1}+4x_{2}=5. \end{cases}\end{equation}
In this case $\Gamma_{1}=\left(\begin{array}{cccc}
2 & -1\\
4 & 1 \end{array}\right)$,
$\Gamma_{2}=\left(\begin{array}{cccc}
-1 & 2\\
1 & 4 \end{array}\right)$ and $\det\Gamma_{1}\cdot\det\Gamma_{2}\neq0$.
 Then
 $$\Gamma_{1}^{-1}=\left(\begin{array}{cccc}
\frac{1}{6} & \frac{1}{6}\\[2mm]
-\frac{2}{3} & \frac{1}{3} \end{array}\right) \ \ \mbox{and} \ \
\Gamma_{2}^{-1}=\left(\begin{array}{cccc}
-\frac{2}{3} & \frac{1}{3}\\[2mm]
\frac{1}{6} & \frac{1}{6} \end{array}\right).$$
The condition (\ref{umumiyshart1}) is satisfied, i.e.,
$\Gamma_{1}^{-1}\cdot a_{1}=\Gamma_{2}^{-1}\cdot a_{2}$, where $a_{1}=\left(\begin{array}{cc}1\\5\end{array}\right)$,
 $a_{2}=\left(\begin{array}{cc}1\\5\end{array}\right)$.
 Consequently, system of equations (\ref{misol3}) has a solution $x_{1}=x_{2}=1$.
\end{example}

For an evolution algebra $E$ introduce the following sequences
$$E^{[1]}=E^{<1>}=E, \ \ E^{[k+1]}=E^{[k]}E^{[k]}, \ \ E^{<k+1>}=E^{<k>}E, \ \ k\geq 1.$$

\begin{equation} E^{k}=\sum_{i=1}^{k-1}E^{i}E^{k-i}.
\end{equation}

The following inclusions hold
$$E^{<k>}\subseteq E^{k},\,\,\,\,\,E^{[k+1]}\subseteq E^{2^k}.$$

Since $E$ is commutative algebra we obtain $$E^{k}=\sum_{i=1}^{\lfloor k/2\rfloor}E^{i}E^{k-i},$$
where $\lfloor x\rfloor$ denotes the integer part of $x$.

\begin{definition} An evolution algebra $E$ is called right nilpotent if there exists some $s\in \mathbb{N}$ such that
$E^{<s>}=0$. The smallest $s$ such that $E^{<s>}=0$ is called the index of right nilpotency.
\end{definition}

\begin{definition}
An evolution algebra $E$ is called nilpotent if there exists some $n\in \mathbb{N}$ such that $E^n=0$.
The smallest $n$ such that $E^n=0$ is called the index of nilpotency.
\end{definition}
In \cite{12}, it is proved that the notions of nilpotent and right nilpotent are equivalent.

\begin{definition}
An algebra $A$ is called solvable if there exits some $t\in \mathbb{N}$ such that $A^{[t]}=0$. The
smallest $t$ such that $A^{[t]}=0$ is called the index of solvability.
\end{definition}

In \cite{13} it is proved that the notions of nil and right nilpotency are equivalent for evolution algebras.
Moreover, the matrix of structural constants for such an algebra has upper (or lower, up to permutation
of basis elements of the algebra) triangular form, i.e. the following results are known:

\begin{theorem}\label{nilpotentlik} (\cite{13}) The following statements are equivalent for an n-dimensional evolution algebra $E$:\\
(a) The matrix corresponding to $E$ can be written as $$M_{E}=\left(\begin{array}{cccccccc}
0 & a_{12} & a_{13} & \dots & a_{1n}\\
0 & 0 & a_{23} & \dots & a_{2n}\\
0 & 0 & 0 & \dots & a_{3n}\\
\vdots & \vdots & \vdots & \cdots & \vdots\\
0 & 0 & 0 & \cdots & 0\\
\end{array}\right);$$\\
(b) $E$ is a right nilpotent algebra;\\
(c) $E$ is a nil algebra.
\end{theorem}

Take an evolution algebra $E$ as algebra $\mathcal A$, and by (\ref{svyaz}) define $E_x$, then
the following question naturally arises: What about right nilpotency
of the algebra $E_{x}$ if the algebra $E$ is right nilpotent?

\begin{theorem}\label{nilevolution}
Let $E$ be any right nilpotent $n$-dimensional real (complex) evolution algebra,
then for any $x\in E$ the evolution algebra $E_{x}$ is a right nilpotent.
\end{theorem}
\begin{proof}
Since $E$ is right nilpotent $n$-dimensional real (complex) evolution algebra,
then its matrix of structural constants (by Theorem \ref{nilpotentlik}) is
$$M_{E}=\left(\begin{array}{cccccccc}
0 & a_{12} & a_{13} & \dots & a_{1n}\\
0 & 0 & a_{23} & \dots & a_{2n}\\
0 & 0 & 0 & \dots & a_{3n}\\
\vdots & \vdots & \vdots & \cdots & \vdots\\
0 & 0 & 0 & \cdots & 0\\
\end{array}\right).$$
Consequently the matrix of structural constants of the evolution algebra $E_{x}$ has a form
$$M_{E_{x}}=\left(\begin{array}{cccccccc}
0 & 0 & 0& \dots & 0\\
2a_{12}x_{1} & 0 & 0& \dots & 0\\
2a_{13}x_{1} & 2a_{23}x_{2} & 0 & \dots & 0\\
\vdots & \vdots & \vdots & \cdots & \vdots\\
2a_{1n}x_{1} & 2a_{2n}x_{2} & 2a_{3n}x_{3} & \cdots & 0\\
\end{array}\right).$$
Now we give change of basis as the following
$$e_{i}=\tilde e_{n-(i-1)}, \ \ i=\overline{1,n}.$$
Then this evolution algebra is
isomorphic to the evolution algebra which matrix of structural constants
in a view upper-triangular form:
$$M_{E_{x}}=\left(\begin{array}{cccccccc}
0 & 2a_{n-1,n}x_{n-1} & \cdots & 2a_{3n}x_{3} & 2a_{2n}x_{2} & 2a_{1n}x_{1}\\
\vdots & \vdots & \vdots & \cdots & \vdots & \vdots\\
0 & 0 & \cdots & 0 & 2a_{23}x_{2} & 2a_{13}x_{1}\\
0 & 0 & \cdots & 0 & 0 & 2a_{12}x_{1}\\
0 & 0 & \cdots & 0 & 0 & 0\\
\end{array}\right).$$
This means that the evolution algebra $E_{x}$ is right nilpotent (see Theorem \ref{nilpotentlik}).
\end{proof}

The following example shows that the opposite of the theorem is not always appropriate.

\begin{example} Let  $\left(\begin{array}{cccc}
a & b\\
c & 0
\end{array}\right)$ be a matrix of structural constants of the 2-dimensional evolution algebra $E$.
Then matrix of structural constants of the evolution algebra $E_{x}$ is:
$M_{E_{x}}=\left(\begin{array}{cccc}
2ax_{1} & 2cx_{2}\\
2bx_{1} & 0
\end{array}\right).$
If we take $x_{1}=0$, then  $M_{E_{x}}=\left(\begin{array}{cccc}
0 & 2cx_{2}\\
0 & 0
\end{array}\right)$ is in a view upper-triangular form, that means
the evolution algebra $E_{x}$ is nilpotent. But
the evolution algebra $E$ is non-nilpotent if $bc\neq0$.
\end{example}
\begin{remark} Let two real evolution algebras $E$ and $E'$ are isomorphic. Then corresponding evolution algebras
$E_{x}$ and $E'_{y}$ may be non-isomorphic for any nonzero $x$ and $y$, in general.
\end{remark}
\begin{example}
 Let $E$ and $E'$ two real evolution algebras with matrix of structural constants
 $\left(\begin{array}{cccc}
 1 & 2\\
 3 & 4 \end{array}\right)$ and
 $\left(\begin{array}{cccc}
 a & \frac{3a^2}{16d}\\
 \frac{8d^2}{a} & {d}\end{array}\right)$ respectively, where $ad\neq0$.
 These algebras are isomorphic and isomorphism is $\left(\begin{array}{cccc}
 0 & \frac{1}{d}\\
 \frac{4}{a} & 0 \end{array}\right)$. Thus for evolution algebras $E$ and $E'$ we denote by $E_{x}$ and $E_{y}$
 evolution algebras with matrixes of structural constants $\left(\begin{array}{cccc}
 x_{1} & 3x_{2}\\
 2x_{1} & 4x_{2} \end{array}\right)$ and $\left(\begin{array}{cccc}
 ay_{1} & \frac{8d^2}{a}y_{2}\\
 \frac{3a^2}{16d}y_{1} & dy_{2} \end{array}\right)$, respectively. These evolution algebras constructed by (\ref{svyaz})
 for algebras $E$ and $E'$, respectively.
We assume that evolution algebras $E_{x}$ and $E_{y}$ are isomorphic.
Then there exists an isomorphism $\varphi=\left(\begin{array}{cccc}
 \alpha & \beta\\
 \gamma & \delta \end{array}\right)$ such that ${\rm det}\varphi\neq0$.

 Let ${e_{1},e_{2}}$ and ${f_{1},f_{2}}$ be bases of evolution algebras $E_{x}$ and $E_{y}$, respectively.
 Since $x,y$ are non-zero, $\varphi(e_{1}e_{2})=0$ implies that $\alpha\gamma=\beta\delta=0$.

Since ${\rm det}\varphi\neq0$ we don't consider the cases $\alpha=\beta=0$ and
$\gamma=\delta=0$. So we have only the following cases:\\
\textbf{Case $\alpha=\delta=0$}. In this case $\varphi(e_{1})=\beta f_{2}$ and $\varphi(e_{2})=\gamma f_{1}$.
From $\varphi(e_{1}e_{1})=\varphi(e_{1})\varphi(e_{1})$ and $\varphi(e_{2}e_{2})=\varphi(e_{2})\varphi(e_{2})$ we get
\begin{equation}\label{example5}\begin{array}{llll}
3\gamma x_{2}=\beta^2y_{1}\frac{3a^2}{16d}\\[2mm]
x_{1}\beta=\beta^2y_{2}d\\[2mm]
4\gamma x_{2}=\gamma^2y_{1}a\\[2mm]
2\beta x_{1}=\gamma^2y_{2}\frac{8d^2}{a}.
\end{array}\end{equation}
From first and third relations of (\ref{example5}) we get $\beta^2a=4\gamma^2d$.
Similarly, from second and fourth relations of (\ref{example5}) we get also $\beta^2a=4\gamma^2d$.
But this equation doesn't have non-zero real solutions $\beta$ and $\gamma$ if $ad<0$.\\
\textbf{Case $\beta=\gamma=0$}. Similarly to the above case we get the equation $3\delta^2a=32\alpha^2d$.
This equation also doesn't have non-zero real solutions if $ad<0$.

Thus an isomorphism $\varphi$
doesn't exist and so evolution algebras $E_{x}$ and $E_{y}$ are not isomorphic for any non-zero $x$ and $y$.
\end{example}
In the following section we give another approximation of algebras, and prove that for two isomorphic evolution
algebras there are their approximated algebras which are isomorphic two.

\section{Transposed evolution algebras}

In this section we consider evolution algebras corresponding to a given algebra as follows:
matrix of structural constants of the evolution algebra is
transposed matrix of $J_{F}(x)$ which defined in section \ref{sec:two}, i.e.
matrix of considering evolution algebra is $(\beta_{pk}(x))_{p,k=1}^n$, where
\begin{equation}\label{transposedstructure}
\beta_{pk}(x)=\sum_{i}(\gamma_{ki,p}+\gamma_{ik,p})x_{i}.
\end{equation}
Formed evolution algebra we call an \textit{approximation of given algebra $A$ by transposed}
and denote by $E_{A_{x}}^T$.
\begin{remark} Approximation of a Lie algebra\footnote{See https://en.wikipedia.org/wiki/Lie$_-$algebra} by (\ref{svyaz}) (and by (\ref{transposedstructure})) is an abelian evolution algebra.
\end{remark}
\begin{proposition}
Let $A$ and $B$ be $n$-dimensional isomorphic evolution algebras. Then there exist $x\in A$ and $y\in B$
such that $E_{A_{x}}^T\simeq E_{B_{y}}^T$.
\end{proposition}
\begin{proof}
Let matrixes of structural constants of evolution algebras $A$ and $B$ has the following forms:
$\left(\begin{array}{ccccc}
a_{11} & a_{12} & \dots & a_{1n}\\
a_{21} & a_{22} & \dots & a_{2n}\\
\vdots & \vdots & \vdots & \vdots\\
a_{n1} & a_{n2} & \dots & a_{nn}\end{array}\right)$ and
$\left(\begin{array}{ccccc}
b_{11} & b_{12} & \dots & b_{1n}\\
b_{21} & b_{22} & \dots & b_{2n}\\
\vdots & \vdots & \vdots & \vdots\\
b_{n1} & b_{n2} & \dots & b_{nn}\end{array}\right)$ respectively.
Then matrixes of structural constants of  corresponding evolution
 algebras $E_{A_{x}}^T$ and $E_{B_{y}}^T$ are
$$J_{F_{A}}(x)=\left(\begin{array}{ccccc}
2a_{11}x_{1} & 2a_{12}x_{1} & \dots & 2a_{1n}x_{1}\\
2a_{21}x_{2} & 2a_{22}x_{2} & \dots & 2a_{2n}x_{2}\\
\vdots & \vdots & \vdots & \vdots\\
2a_{n1}x_{n} & 2a_{n2}x_{n} & \dots & 2a_{nn}x_{n}\end{array}\right),$$
$$J_{F_{B}}(y)=\left(\begin{array}{ccccc}
2b_{11}y_{1} & 2b_{12}y_{1} & \dots & 2b_{1n}y_{1}\\
2b_{21}y_{2} & 2b_{22}y_{2} & \dots & 2b_{2n}y_{2}\\
\vdots & \vdots & \vdots & \vdots\\
2b_{n1}y_{n} & 2b_{n2}y_{n} & \dots & 2b_{nn}y_{n}\end{array}\right)$$ respectively.
If we take $x_{i}=c$ and $y_{i}=d$, $i=1,2,\dots, n$ then we get evolution algebras which
isomorphic to given evolution algebras $A$ and $B$ respectively.
\end{proof}

\section{Approximations of 2 and 3-dimensional Leibniz algebras}

Let us introduce some definitions and notations, all of them necessary for the understanding of this
section.
\begin{definition} A Leibniz algebra over $K$ is a vector space $\mathcal{L}$ equipped with a bilinear map, called
bracket,
\begin{quote}
$[-,-]:\mathcal{L}\times\mathcal{L}\rightarrow\mathcal{L}$
\end{quote}
satisfying the Leibniz identity:
\begin{quote}
$[x,[y,z]]=[[x,y],z]-[[x,z],y]$,
\end{quote}
for all $x,y,z\in\mathcal{L}$.
\end{definition}

For a given Leibniz algebra $(\mathcal{L},[-,-])$ we define lower central series as follows:
$$\mathcal{L}^1=\mathcal{L}, \mathcal{L}^{k+1}=[\mathcal{L}^{k},\mathcal{L}], k\geq1.$$
\begin{definition}
A Leibniz algebra $\mathcal{L}$ said to be nilpotent, if there exists $n\in\mathbb{N}$ such that $\mathcal{L}^n=0$.
The minimal number $n$ such property is said to be the index nilpotency of the algebra $L$.
\end{definition}

The classification of the two and three dimensional nilpotent
Leibniz algebras was obtained in \cite{14} and \cite{15}, respectively. The following theorems show these
classifications.

\begin{theorem}\label{Leibniznil2}
Let $\mathcal{L}$ be a 2-dimensional nilpotent Leibniz algebra. Then $\mathcal{L}$ is an abelian algebra or it is isomorphic to
\begin{quote} $\mu_{1}:[e_{1},e_{1}]=e_{2}$.
\end{quote}
\end{theorem}

\begin{theorem}\label{Leibniznil3}
Let $\mathcal{L}$ be a 3-dimensional nilpotent Leibniz algebra. Then $\mathcal{L}$ is isomorphic to one of the
following pairwise non-isomorphic algebras:
\begin{quote} $\lambda_{1}:$ abelian,\\
$\lambda_{2}: [e_{1},e_{1}]=e_{3}$,\\
$\lambda_{3}: [e_{1},e_{2}]=e_{3}$, $[e_{2},e_{1}]=-e_{3}$,\\
$\lambda_{4}(\alpha): [e_{1},e_{1}]=e_{3}$,$[e_{2},e_{2}]=\alpha e_{3}$, $[e_{1},e_{2}]=e_{3}$\\
$\lambda_{5}: [e_{2},e_{1}]=e_{3}$, $[e_{1},e_{2}]=e_{3}$,\\
$\lambda_{6}: [e_{1},e_{1}]=e_{2}$, $[e_{2},e_{1}]=e_{3}$.
\end{quote}
\end{theorem}

\begin{proposition}\label{prop2-dimLeib}
Approximation by transposed of the 2-dimensional nilpotent Leibniz algebra in any point $x\in\mathcal{L}$ is also nilpotent evolution algebra.
\end{proposition}
\begin{proof} Due Theorem \ref{Leibniznil2} it is easy to see that for abelian Leibniz algebra
corresponds abelian evoution algebra which structural constants obtained by (\ref{transposedstructure}).

For Leibniz algebra $\mu_{1}$ we define evolution algebra $E_{\mu_{1}, x}^T$ which structural constants defined by
(\ref{transposedstructure}):
\begin{quote}$\beta_{11}(x)=(\gamma_{11,1}+\gamma_{11,1})x_{1}+(\gamma_{12,1}+\gamma_{21,1})x_{2}=0$,\\
$\beta_{12}(x)=(\gamma_{21,1}+\gamma_{12,1})x_{1}+(\gamma_{22,1}+\gamma_{22,1})x_{2}=0$,\\
$\beta_{21}(x)=(\gamma_{11,2}+\gamma_{11,2})x_{1}+(\gamma_{12,2}+\gamma_{21,2})x_{2}=2x_{1}$,\\
$\beta_{22}(x)=(\gamma_{21,2}+\gamma_{12,2})x_{1}+(\gamma_{22,2}+\gamma_{22,2})x_{2}=0$,\\
\end{quote}
where $\gamma_{ij,k}$-structural constants of nilpotent Leibniz algebra.\\
Therefore matrix of structural constants of the evolution algebra $E_{\mu_{1},{x}}^T$ is:
$\left(\begin{array}{cccc}
0 & 0\\
2x_{1} & 0\end{array}\right)$. By Theorems \ref{nilevolution} and \ref{nilpotentlik} the evolution algebra $E_{\mu_{1}, {x}}^T$ is right nilpotent.
\end{proof}

\begin{proposition}
Approximation by transposed of the 3-dimensional nilpotent Leibniz algebra in any point $x\in\mathcal{L}$ is also nilpotent evolution algebra.
\end{proposition}
\begin{proof} Due Theorem \ref{Leibniznil3} it is easy to see that for abelian Leibniz algebra $\lambda_{1}$
corresponds abelian evolution algebra (in particular, nilpotent) with structural constants obtained by (\ref{transposedstructure}).

For Leibniz algebra $\lambda_{2}$ we define evolution algebra $E_{\lambda_2, {x}}^T$ which structural constants defined by
(\ref{transposedstructure}):
\begin{quote}
$\beta_{11}(x)=(\gamma_{11,1}+\gamma_{11,1})x_{1}+(\gamma_{12,1}+\gamma_{21,1})x_{2}+(\gamma_{13,1}+\gamma_{31,1})x_{3}=0$,\\
$\beta_{12}(x)=(\gamma_{21,1}+\gamma_{12,1})x_{1}+(\gamma_{22,1}+\gamma_{22,1})x_{2}+(\gamma_{23,1}+\gamma_{32,1})x_{3}=0$,\\
$\beta_{13}(x)=(\gamma_{31,1}+\gamma_{13,1})x_{1}+(\gamma_{32,1}+\gamma_{23,1})x_{2}+(\gamma_{33,1}+\gamma_{33,1})x_{3}=0$,\\
$\beta_{21}(x)=(\gamma_{11,2}+\gamma_{11,2})x_{1}+(\gamma_{12,2}+\gamma_{21,2})x_{2}+(\gamma_{13,2}+\gamma_{31,2})x_{3}=0$,\\
$\beta_{22}(x)=(\gamma_{21,2}+\gamma_{12,2})x_{1}+(\gamma_{22,2}+\gamma_{22,2})x_{2}+(\gamma_{23,2}+\gamma_{32,2})x_{3}=0$,\\
$\beta_{23}(x)=(\gamma_{31,2}+\gamma_{13,2})x_{1}+(\gamma_{32,2}+\gamma_{23,2})x_{2}+(\gamma_{33,2}+\gamma_{33,2})x_{3}=0$,\\
$\beta_{31}(x)=(\gamma_{11,3}+\gamma_{11,3})x_{1}+(\gamma_{12,3}+\gamma_{21,3})x_{2}+(\gamma_{13,3}+\gamma_{31,3})x_{3}=2x_{1}$,\\
$\beta_{32}(x)=(\gamma_{21,3}+\gamma_{12,3})x_{1}+(\gamma_{22,3}+\gamma_{22,3})x_{2}+(\gamma_{23,3}+\gamma_{32,3})x_{3}=0$,\\
$\beta_{33}(x)=(\gamma_{31,3}+\gamma_{13,3})x_{1}+(\gamma_{32,3}+\gamma_{23,3})x_{2}+(\gamma_{33,3}+\gamma_{33,3})x_{3}=0$,\\
\end{quote}
where $\gamma_{ij,k}$-structural constants of nilpotent Leibniz algebra $\lambda_{2}$.\\
Therefore matrix of structural constants of the evolution algebra $E_{\lambda_{2}, {x}}^T$ is:
$\left(\begin{array}{cccccc}
0 & 0 & 0\\
0 & 0 & 0\\
2x_{1} & 0 & 0\end{array}\right)$. By Theorems \ref{nilevolution} and \ref{nilpotentlik} the
evolution algebra  is right nilpotent.

For algebras $\lambda_{3}$, $\lambda_{4}(\alpha)$, $\lambda_{5}$ and $\lambda_{6}$ similarly
we can prove that corresponding evolution algebras
$E_{\lambda_{3}, x}^T$, $E_{\lambda_{4}(\alpha),x}^T$, $E_{\lambda_{5}, x}^T$ and
$E_{\lambda_{6}, x}^T$ are also nilpotent.
\end{proof}

Let $E$ be a $2$-dimensional evolution algebra over the field of real numbers.
Such algebras are classified in \cite{7}, \cite{8}:
\begin{theorem}\label{thmurodov}
Any two-dimensional real evolution algebra $E$ is isomorphic to one of the following
pairwise non-isomorphic algebras:

(i) $dim(E^{2})=1$:

$E_{1}$: $e_{1}e_{1}=e_{1}$, $e_{2}e_{2}=0$;

$E_{2}$: $e_{1}e_{1}=e_{1}$, $e_{2}e_{2}=e_{1}$;

$E_{3}$: $e_{1}e_{1}=e_{1}+e_{2}$, $e_{2}e_{2}=-e_{1}-e_{2}$;

$E_{4}$: $e_{1}e_{1}=e_{2}$, $e_{2}e_{2}=0$;

$E_{5}$: $e_{1}e_{1}=e_{2}$, $e_{2}e_{2}=-e_{2}$;

(ii) $dim(E^{2})=2$:

$E_{6}(a_{2};a_{3})$: $e_{1}e_{1}=e_{1}+a_{2}e_{2}$, $e_{2}e_{2}=a_{3}e_{1}+e_{2}$;
$1-a_{2}a_{3}\neq0$, $a_{2},a_{3}\in\mathbb{R}$. Moreover $E_{6}(a_{2};a_{3})$
is isomorphic to  $E_{6}(a_{3};a_{2})$.

$E_{7}(a_{4})$: $e_{1}e_{1}=e_{2}$, $e_{2}e_{2}=e_{1}+a_{4}e_{4}$, where $a_{4}\in\mathbb{R}$;
\end{theorem}

We have following propositions.
\begin{proposition}
The evolution algebra $E_{\mu_{1}, x}^T$ given in Proposition \ref{prop2-dimLeib} is\\
i) an abelian if $x_{1}=0$;\\
ii) isomorphic to $E_{4}$ if $x_{1}\neq0$.
\end{proposition}
\begin{proof} i) Straightforward.

ii) Taking change of basis like $e_{1}'=\frac{1}{2x_{1}}e_{2},\,\,\,e_{2}'=\frac{1}{2x_{1}}e_{1}$ completes the proof.
\end{proof}

\begin{proposition}The followings hold:\\
\begin{itemize}
\item[1)] the evolution algebra $E_{\lambda_{2}, x}^{T}$ is isomorphic to evolution algebra with matrix of structural constants
$\left(\begin{array}{cccccc}
1 & 0 & 0\\
0 & 0 & 0\\
0 & 0 & 0\end{array}\right)$ if $x_{1}\neq0$;\\
\item[2)] the evolution algebra $E_{\lambda_{3}, x}^{T}$ is an abelian;\\
\item[3)] the evolution algebra $E_{\lambda_{4}(\alpha), x}^{T}$ is isomorphic to evolution algebra with matrix of structural constants\\
\begin{itemize}\item[3.1)] {$\left(\begin{array}{cccccc}
0 & 1 & 1\\
0 & 0 & 0\\
0 & 0 & 0\end{array}\right)$ if $(2x_{1}+x_{2})(x_{1}+2\alpha x_{2})\neq0$;}\\
\item[3.2)] {$\left(\begin{array}{cccccc}
0 & 1 & 0\\
0 & 0 & 0\\
0 & 0 & 0\end{array}\right)$ if $2x_{1}+x_{2}=0$ or $x_{1}+2\alpha x_{2}=0$;}
\end{itemize}
\item[4)] the evolution algebra $E_{\lambda_{5}, x}^{T}$ is isomorphic to evolution
algebra with matrix of structural constants\\
\begin{itemize}\item[4.1)] {$\left(\begin{array}{cccccc}
0 & 1 & 1\\
0 & 0 & 0\\
0 & 0 & 0\end{array}\right)$ if $x_{1}x_{2}\neq0$;}\\
\item[4.2)] {$\left(\begin{array}{cccccc}
0 & 1 & 0\\
0 & 0 & 0\\
0 & 0 & 0\end{array}\right)$ if $x_{1}=0$ or $x_{2}=0$;}
\end{itemize}
\item[5)] the evolution algebra $E_{\lambda_{6}, x}^{T}$ is isomorphic to evolution
algebra with matrix of structural constants \\
\begin{itemize}\item[5.1)] {$\left(\begin{array}{cccccc}
0 & 1 & 1\\
0 & 0 & 0\\
0 & 1 & 0\end{array}\right)$ if $x_{1}x_{2}>0$;}\\
\item[5.2)] {$\left(\begin{array}{cccccc}
0 & -1 & -1\\
0 & 0 & 0\\
0 & 1 & 0\end{array}\right)$ if $x_{1}x_{2}<0$;}
\item[5.3)] {$\left(\begin{array}{cccccc}
1 & 0 & 0\\
0 & 0 & 0\\
0 & 0 & 0\end{array}\right)$ if $x_{1}=0$;}\\
\item[5.4)] {$\left(\begin{array}{cccccc}
0 & 1 & 0\\
0 & 0 & 1\\
0 & 0 & 0\end{array}\right)$ if $x_{2}=0$;}
\end{itemize}
\end{itemize}
\end{proposition}

\begin{proof}
\begin{itemize}
\item[1)] $\left(\begin{array}{cccccc}
0 & 0 & 0\\
0 & 0 & 0\\
2x_{1} & 0 & 0\end{array}\right)$ is matrix of structural constants of the evolution
algebra $E_{\lambda_{2}, x}^{T}$. By changing of basis like
$$e_{1}'=\frac{1}{2x_{1}}e_{3},\,\,\, e_{2}'=e_{2},\,\,\, e_{3}'=\frac{1}{2x_{1}}e_{1}$$
we give proof of 1);\\
\item[2)] Clearly;\\
\item[3)] {$\left(\begin{array}{cccccc}
0 & 0 & 0\\
0 & 0 & 0\\
2x_{1}+x_{2} & x_{1}+2\alpha x_{2} & 0\end{array}\right)$ is matrix of structural constants of the evolution
algebra $E_{\lambda_{4}(\alpha), x}^{T}$. If $(2x_{1}+x_{2})(x_{1}+2\alpha x_{2})\neq0$
by changing of basis like
$$e_{1}'=e_{3},\,\,\, e_{2}'=(x_{1}+2\alpha x_{2})e_{2},\,\,\, e_{3}'=(2x_{1}+x_{2})e_{1}$$
we give proof of 3.1);\\
If $2x_{1}+x_{2}=0$ then by changing of basis like
$$e_{1}'=\frac{1}{(1-4\alpha)x_{1}}e_{3}, \ \ e_{2}'=\frac{1}{(1-4\alpha)x_{1}}e_{2}, \ \ e_{3}'=e_{1}$$
and if $x_{1}+2\alpha x_{2}=0$ by changing of basis like
$$e_{1}'=\frac{1}{(1-4\alpha)x_{2}}e_{3}, \ \ e_{2}'=\frac{1}{(1-4\alpha)x_{2}}e_{1}, \ \ e_{3}'=e_{2}$$ we give proof of 3.2);}\\

\item[4)]{It is similar to above case;}\\
\item[5)]{$\left(\begin{array}{cccccc}
0 & 0 & 0\\
2x_{1} & 0 & 0\\
x_{2} & x_{1} & 0\end{array}\right)$ is matrix of structural constants of the evolution
algebra $E_{\lambda_{6}, x}^{T}$.\\
If $x_{1}x_{2}>0$, we give change of basis:
$$e_{1}'=\frac{x_{2}}{x_{1}\sqrt{2x_{1}x_{2}}}e_{3}, \ \ e_{2}'=\frac{x_{2}^2}{2x_{1}^3}e_{1}, \ \
e_{3}'=\frac{x_{2}^2}{2x_{1}^2 x_{2}}e_{2}.$$
Therefore we have proof of 5.1).\\
If $x_{1}x_{2}<0$, we give change of basis:
$$e_{1}'=\frac{x_{2}}{x_{1}\sqrt{-2x_{1}x_{2}}}e_{3}, \ \ e_{2}'=\frac{x_{2}^2}{2x_{1}^3}e_{1}, \ \ e_{3}'=\frac{x_{2}^2}{2x_{1}^2 x_{2}}e_{2}.$$
Therefore we have proof of 5.2).\\
If $x_{1}=0$, this case is similar to 1).\\
If $x_{2}=0$, then $\left(\begin{array}{cccccc}
0 & 0 & 0\\
2x_{1} & 0 & 0\\
0 & x_{1} & 0\end{array}\right)$ is matrix of structural constants of the evolution
algebra $E_{\lambda_{6}, x}^{T}$. By changing of basis
$$e_{1}'=\frac{1}{\sqrt{2}x_{1}}e_{3}, \ \ e_{2}'=\frac{1}{2x_{1}}e_{2}, \ \ e_{3}'=\frac{1}{2x_{1}}e_{1}$$
one completes the proof of 5.4)}.
\end{itemize}
\end{proof}

\begin{remark} Another type of approximation of an algebra can be given as follows.

Let $\mathcal A$ be an $n$-dimensional algebra
over the field of complex (or real) numbers.

Denote by $M_{\mathcal A}$ the matrix of structural constants of $\mathcal A$.
In case of complex field the matrix $M_{\mathcal A}$ can be considered as a vector in $\mathbb C^{n^3}$.
Consider a norm of $M=(\gamma_{ijk})\in \mathbb C^{n^3}$ defined by
$$\|M\|=\max_{i,j,k}|\gamma_{ijk}|.$$

An algebra  $\mathcal B$  is called $\varepsilon$-approximation
of an algebra $\mathcal A$ if
$$\|M_{\mathcal A}-M_{\mathcal B}\|\leq \varepsilon.$$

It is clear that if $\mathcal B$ is $\varepsilon$-approximation for $\mathcal A$ then
$\mathcal A$ is $\varepsilon$-approximation for $\mathcal B$.

The interesting question is ``How properties of $\mathcal A$ depends on properties of $\mathcal B$?"

Let $\mathcal A^{[s,t]}$, $s, t\geq 0$ be a flow of $n$-dimensional algebras
over the field of complex (or real) numbers (see \cite{LR4}, \cite{LR3}).
Each flow of algebras can be considered as a continuous time dynamical system such that
in each fixed pair $(s,t)$ of time
its state is a finite-dimensional algebra.
In such a dynamical system one wants to know what
will be the limit algebra (if it exists):
\begin{equation}\label{aw}
\lim_{t-s\to+\infty}\mathcal A^{[s,t]}=\mathcal A.
\end{equation}
In case the limit does not exist then one asks about possible limit algebras.

Assume the limit (\ref{aw}) exists, that is for any $\varepsilon>0$ there exists
$T=T(\varepsilon)>0$ such that for any $t,s$ with $|t-s|>T$ we have
$$\|M_{\mathcal A^{[s,t]}}-M_{\mathcal A}\|\leq \varepsilon.$$

Thus the limit algebra $\mathcal A$ can be considered as  $\varepsilon$-approximation
for each $\mathcal A^{[s,t]}$ with $|t-s|>T$.
\end{remark}

\end{document}